\documentclass[12pt]{amsart}

\usepackage[dvips]{graphics}

    \usepackage{amsmath,amsfonts,eucal,xy, enumerate, amscd, amssymb}
\xyoption{all}

     \newtheorem{theorem}{Theorem}[section]
     \newtheorem{proposition}[theorem]{Proposition}
     
     \newtheorem{corollary}[theorem]{Corollary}
     \newtheorem{lemma}[theorem]{Lemma}

 \theoremstyle{definition}

     \newtheorem{examples}[theorem]{Examples}
     \newtheorem{definition}[theorem]{Definition}

     \newcommand{\In}{\subseteq} 
\newcommand{\script}[1]{\mathcal #1}

\renewcommand{\Bbb}[1]{\mathbb #1}

\newcommand{\bk}{{\Bbb K}}

\newcommand{\cM}{{\mathcal M}}
\newcommand{\cP}{{\mathcal P}}

\newcommand{\fin}{{\operatorname{fin}}}

\newcommand{\G}{\Gamma}

\renewcommand{\frak}{\mathfrak}
\newcommand{\gA}{{\frak A}}
\newcommand{\gB}{{\frak B}}
\newcommand{\gC}{{\frak C}}
\newcommand{\gF}{{\frak F}}
\newcommand{\gK}{{\frak K}}
\newcommand{\gL}{{\frak L}}

\newcommand{\gN}{{\frak N}}
\newcommand{\gP}{{\frak P}}
\newcommand{\gW}{{\frak W}}
\newcommand{\inj}{{\operatorname{inj}}}
\newcommand{\lam}{\lambda}

\newcommand{\ori}{{\operatorname{ori}}}

\newcommand{\po}{^\circ}

\newcommand{\s}{\sigma}
\newcommand{\sB}{{\script B}}

\newcommand{\sC}{{\script C}}

\newcommand{\sS}{{\script S}}

\newcommand{\sM}{{\script M}}
\newcommand{\sO}{{\script O}}
\newcommand{\sT}{{\script T}}

\newcommand{\sur}{{\operatorname{sur}}}

\newcommand{\von}{{\operatorname{von}}}

     \newcommand{\ie}{{i.e., }}
     \newcommand{\eg}{{e.g.\  }}

\begin{document}

\parskip=0.5\baselineskip
\baselineskip=1.44\baselineskip

\title
{Operator Ideals arising from Generating Sequences}
\date{January 11, 2011; for the Proceedings of ICA 2010 (World Scientific).
}
\author{Ngai-Ching Wong}

\address{
Department of Applied Mathematics, National Sun Yat-sen
  University, and National Center for Theoretical Sciences, Kaohsiung, 80424, Taiwan, R.O.C.}

\email{wong@math.nsysu.edu.tw}

\dedicatory{Dedicated to Professor Kar-Ping Shum in honor of his seventieth birthday}

\thanks{Partially supported  by a
Taiwan National Science Council grant 99-2115-M-110-007-MY3}



\begin{abstract}
 In this note, we will discuss how to relate an operator ideal on Banach spaces to the sequential
 structures it defines.  Concrete examples of ideals of compact, weakly compact, completely continuous,
 Banach-Saks and weakly Banach-Saks operators will be demonstrated.
 \end{abstract}

\maketitle

\section{Introduction}

Let $T: E\to F$ be  a linear operator between Banach spaces.
Let $U_E, U_F$ be the closed unit balls of $E,F$, respectively.
Note that closed unit balls serve simultaneously the basic model for
open sets and bounded sets in Banach spaces.
The usual way to describe $T$ is to state
either  the bornological property, via $TU_E$, or  the topological
property, via $T^{-1}U_F$, of $T$.
An other way  classifying $T$ is through the sequential structures $T$ preserves.
Here are some well-known examples.

\begin{examples}\label{eg:equiv}
\begin{enumerate}
    \item \quad\ $T$ is {\em bounded} (\ie $TU_E$ is a bounded subset of
$F$)
\\ $\Leftrightarrow$ $T$ is {\em continuous} (\ie $T^{-1}U_F$ is a
0-neighborhood of $E$ in the norm topology)
\\ $\Leftrightarrow$ $T$ is \emph{sequentially bounded} ( i.e.,
$T$ sends bounded sequences to bounded sequences);

    \item \quad\ $T$ is {\em of finite
rank} (\ie $TU_E$ spans a finite dimensional subspace of $F$) \\
$\Leftrightarrow$ $T$ is {\em
weak-norm continuous} (\ie $T^{-1}U_F$ is a 0-neighborhood of $E$ in
the weak topology)
\\ $\Leftrightarrow$ $T$
sends bounded sequences to sequences spanning finite dimensional subspaces of $F$;

    \item \quad\ $T$ is {\em compact} (\ie $TU_E$ is totally
bounded in
 $F$) \\  $\Leftrightarrow$ $T$ is continuous in the topology of uniform
 convergence on norm compact subsets of $E'$
 (\ie $T^{-1}U_F \supseteq K^\circ$, the polar of a norm compact  subset $K$ of
   the dual space $E'$ of $E$)
   \\ $\Leftrightarrow$ $T$ is \emph{sequentially compact} (i.e.,
   $T$ sends bounded sequences to sequences with norm convergent subsequences); and

    \item \quad\ $T$ is {\em weakly compact} (\ie $TU_E$ is relatively weakly compact in
 $F$) \\  $\Leftrightarrow$ $T$ is continuous in the topology of uniform
 convergence on weakly compact subsets of $E'$
 (\ie $T^{-1}U_F \supseteq K^\circ$, the polar of a weakly compact  subset $K$ of
   $E'$)
   \\ $\Leftrightarrow$ $T$ is \emph{sequentially weakly compact} (i.e.,
   $T$ sends bounded sequences to sequences with weakly convergent subsequences).
\end{enumerate}
\end{examples}

In \cite{WW88,WW93, W94}, we  investigate the duality of the topological and bornological
characters of operators demonstrated in the above examples, and applied them in the study of operator ideal
theory in the sense of Pietsch \cite{Pie80}.
In \cite{W08}, we use these concepts in classifying locally convex spaces.
They are also used by other authors in operator algebra theory in \cite{West, CanWest}.

After giving a brief account of the equivalence among the notions of operator ideals, generating
topologies, and generating bornologies developed in \cite{WW88, WW93, W94} in Section 2, we
shall explore into the other option of using sequential structures in Section 3.
The theory is initialed by  Stephani \cite{Step80, Step83}.  We provide a variance here.  In particular, we will show that the notions of operator ideals, generating bornologies and
generating sequences are equivalent in the context of Banach spaces.  However, we note that the sequential
methodology does not seem to be appropriate in the context of locally convex spaces, as successfully as in
\cite{WW88, W08}.  Some examples of popular operator ideals are treated in our new ways as demonstrations in
Section 4.

The author would like to take this opportunity to thank Professor Kar-Ping Shum.  He
has learned a lot from Professor Shum
since he was a student in the Chinese University of Hong Kong in 1980's.
With about 300 publications, Professor Shum has served as a good model for
the author and other fellows since then.
The current work is based  partially on
the thesis of the author finished when he was studying with Professor
Shum.

\section{The triangle of operators, topologies and bornologies}

Let $X$ be a vector space over $\mathbb{K}=\mathbb{R}$ or $\mathbb C$.
Following Hogbe-Nlend \cite{HN77}, by
a {\it vector bornology\/} on $X$ we mean a family $\sB$ of subsets
of $X$ satisfies the following conditions:
\begin{enumerate}
\item[(VB$_1$)]$X=\cup \sB$;
\item[(VB$_2$)]if $B\in \sB$ and $A \subseteq B$ then $A\in \sB$;
\item[(VB$_3$)]$B_{1}+B_{2}\in \sB$ whenever $B_{1},B_{2}\in \sB$;
\item[(VB$_4$)]$\lam B\in \sB$ whenever $\lam \in \bk$ and $B \in \sB$;
\item[(VB$_5$)]the circle hull $\text{ch}\, B$ of any $B$ in $\sB$ belongs to $\sB$.
\end{enumerate}

Elements in $\sB$ are called $\sB$-bounded sets in $X$.
A vector bornology $\sB$ on $X$ is called a {\it convex bornology\/}
if $\Gamma B\in \sB$  for all $B$ in $\sB$,
where $\Gamma B$ is the absolutely convex hull of $B$. The pair
$(X,\sB)$ is called a {\it convex bornological space\/} which is
denoted by
$X^{\sB}$.

Let $(X,\sT)$ be a locally convex space. For any subset $B$ of $X$,
the {\it $\s$-disked hull\/} of $B$ is defined to be

\begin{center}
$\G_{\s}B=\{\sum^{\infty}_{n=1} \lam_{n}b_{n}:\sum^{\infty}_{n=1}
|\lam_{n}| \leq 1, \lam_{n}\in \bk, b_{n}\in
B, n=1,2,\ldots\}$.
\end{center}
A set $B$ in $X$ is said to be {\it $\s$-disked\/} if $B=\G_{\s}B$.
An absolutely convex, bounded subset $B$ of $X$ is said to be {\it
infracomplete} (or a {\it Banach disk\/}) if the normed space
$(X(B),r_B)$ is complete, where
$$
X(B)=\bigcup_{n\geq 1}nB \quad\text{and}\quad r_B(x):=\inf\{\lambda >0 : x\in \lambda B\}
$$
is the gauge of $B$ defined on $X(B)$.

\begin{lemma}[{\cite{WW88}}]
Let $(X,\script P)$ be a locally convex space and B an absolutely
convex bounded subset of X.
\begin{enumerate}
\item[(a)]If B is $\s$-disked then B is infracomplete.
\item[(b)]If B is infracomplete and closed then B is $\s$-disked.
\end{enumerate}
\end{lemma}

Let $X,Y$ be locally convex spaces.
Denote by $\sigma(X,X')$  the weak topology
of $X$ with respect to its dual space $X'$, while  $\cP_\ori(X)$ is the original topology of $X$.
We employ the notion $\cM_\fin(Y)$ for the {\em finite dimensional bornology}
of $Y$ which has a basis consisting of all convex hulls
of finite sets.
On the other hand, $\cM_\von(Y)$ is used for the {\em von Neumann bornology} of $Y$
which consists of all $\cP_\ori(Y)$-bounded subsets of $Y$.
Ordering of topologies and
bornologies are induced by set-theoretical inclusion, as usual.  Moreover, we write briefly
$X_{\cP}$ for a vector space $X$ equipped with a locally convex
topology $\cP$ and
$Y^{\cM}$ for a vector space $Y$ equipped with a convex vector bornology
$\cM$.

 Let $X, Y$ be locally convex  spaces. We denote by $\gL(X,Y)$ and
$L^{\times}(X,Y)$ the collection of all linear operators from $X$ into $Y$
which are continuous and (locally) bounded (\ie sending bounded sets to bounded
sets), respectively.

\begin{definition}[see \cite{W08}]
Let $\sC$ be a subcategory of locally convex spaces.
\begin{enumerate}
    \item ({\bf ``Operators''}) A family $\gA = \{\gA(X,Y) : X,Y \in \sC\}$ of algebras of operators associated to each pair of spaces $X$ and $Y$ in $\sC$ is called an {\em operator ideal}
if
    \begin{description}
        \item[OI$_1$]  $\gA(X,Y)$ is a nonzero vector subspace of $\gL(X,Y)$ for all $X$, $Y$ in $\sC$; and
        \item[OI$_2$]  $RTS \in \gA(X_0,Y_0)$ whenever $R \in \gL(Y,Y_0)$,
$T\in\gA(X,Y)$ and $S \in \gL(X_0,X)$ for any $X_0$, $X$, $Y$ and $Y_0$ in $\sC$.
    \end{description}

    \item ({\bf ``Topologies''}) A family $\cP =\{\cP(X) : X\in \sC\}$ of locally convex   topologies
associated to each space $X$ in $\sC$ is called a {\em generating topology} if
    \begin{description}
    \item[GT$_1$] $\sigma(X,X') \In \cP(X) \In \cP_\ori(X)$ for all $X$ in $\sC$; and
    \item[GT$_2$] $\gL(X,Y) \In \gL(X_\cP,Y_\cP)$ for all $X$ and $Y$ in $\sC$.
    \end{description}

    \item ({\bf ``Bornologies''}) A family $\cM=\{\cM(Y): Y \in \sC\}$ of convex vector
    bornologies associated to each space $Y$ in $\sC$ is called a {\em generating bornology} if
    \begin{description}
    \item[GB$_1$]  $\cM_\fin(Y) \In \cM(Y) \In \cM_\von(Y)$ for all $Y$ in $\sC$; and
    \item[GB$_2$]  $\gL(X,Y) \In L^\times(X^\cM,Y^\cM)$ for all $X$ and $Y$ in $\sC$.
    \end{description}
\end{enumerate}
\end{definition}

\begin{definition}\label{defn:vertices}
Let $\gA$ be an operator ideal, $\cP$ a generating topology and $\cM$  a generating bornology on $\sC$.
\begin{enumerate}
    \item ({\bf ``Operators'' $\rightarrow$ ``Topologies''})
For each $X_0$ in $\sC$, the $\gA$--{\em topology} of $X_0$, denoted by $\sT(\gA)(X_0)$, is the projective topology of $X_0$ with respect to the family
$$
\{T \in \gA(X_0,Y):Y\in\sC\}.
$$
  In other words, a seminorm $p$ of $X_0$ is $\sT(\gA)(X_0)$--continuous if and only if there is a $T$ in $\gA(X_0,Y)$ for some $Y$ in
$\sC$ and a continuous seminorm $q$ of $Y$ such that
$$
p(x) \leq q(Tx), \quad\forall x \in X_0.
$$
In this case, we call $p$ an $\gA$--{\em seminorm} of $X_0$.

    \item ({\bf ``Operators'' $\rightarrow$ ``Bornologies''})
For each $Y_0$ in $\sC$, the $\gA$--{\em bornology} of $Y_0$, denoted by $\sB(\gA)(Y_0)$, is the inductive bornology of $Y_0$ with respect to the family
$$
\{T \in \gA(X,Y_0):X\in\sC\}.
$$
  In other words, a subset $B$ of $Y_0$ is $\sB(\gA)(Y_0)$--bounded if and only if there is a $T$ in $\gA(X,Y_0)$ for some $X$ in
$\sC$ and a topologically bounded subset $A$ of $X$ such that
$$
B \In TA.
$$
In this case, we call $B$ an $\gA$--{\em bounded subset} of $Y_0$.

    \item ({\bf ``Topologies'' $\rightarrow$ ``Operators''})
For $X$, $Y$ in $\sC$, let
$$
\sO(\cP)(X,Y) = \gL(X_\cP,Y)
$$
be the vector space of all continuous operators from $X$ into $Y$ which is still continuous with
respect to the $\cP(X)$--topology.

    \item ({\bf ``Bornologies'' $\rightarrow$ ``Operators''})
For $X$, $Y$ in $\sC$, let
$$
\sO(\cM)(X,Y) = \gL(X,Y)\cap L^\times(X,Y^\cM)
$$
be the vector space of all continuous operators from $X$ into $Y$ which send bounded
sets to $\cM(Y)$--bounded sets.

    \item ({\bf ``Topologies'' $\leftrightarrow$ ``Bornologies''})
For $X$, $Y$ in $\sC$, the $\cP^\circ(Y)$--{\em bornology} of $Y$ (resp.\ $\cM^\circ(X)$--{\em topology} of $X$) is
defined to be the bornology (resp.\ topology) polar to
$\cP(X)$ (resp.\ $\cM(Y)$).  More precisely,
\begin{itemize}
    \item a bounded subset $A$ of $Y$ is $\cP^\circ(Y)$--bounded if and only if  its   polar
$A\po$ is a $\cP(Y'_\beta)$--neighborhood of zero; and
    \item a neighborhood $V$ of zero of $X$ is a $\cM^\circ(X)$--neighborhood of zero if and only if
$V^\circ$ is $\cM(X'_\beta)$--bounded.
\end{itemize}
\end{enumerate}
\end{definition}

Here are two examples: the ideals $\gK$ of compact operators and $\gP$ of
absolutely summing operators (see \eg\cite{Pie80}), the generating systems $\cP_{pc}$ of precompact
topologies (see \eg\cite{Ran72}) and $\cP_{pn}$ of prenuclear topologies (see \eg\cite[p.\ 90]{Sch71}),
and the generating systems $\cM_{pc}$ of precompact bornologies
and $\cM_{pn}$ of prenuclear bornologies (see \eg\cite{HN81}),
respectively.

\begin{definition}
An operator ideal $\gA$ on Banach spaces is said to be
\begin{enumerate}
    \item \emph{injective} if $S\in\gA(E,F_0)$ infers
$T\in \gA(E,F)$, whenever  $T\in \gL(E,F)$ and
$\|Tx\|\leq \|Sx\|, \forall x\in E$;
    \item \emph{surjective} if $S\in\gA(E_0, F)$ infers
$T\in \gA(E,F)$, whenever $T\in \gL(E,F)$ and
$TU_E\subseteq SU_{E_0}$.
\end{enumerate}
\end{definition}

The \emph{injective hull} $\gA^\inj$ and the \emph{surjective hull} $\gA^\sur$,
of $\gA$ is the
intersection of all injective and surjective operator ideals
containing $\gA$, respectively.
For example, $\gK=\gK^\inj=\gK^\sur$ and $\gP=\gP^\inj \varsubsetneq\gP^\sur$ (see, e.g., \cite{Pie80}).

In the following, we see that the notions of operator ideals, generating topologies,
  and generating bornologies are equivalent (see also \cite{West} for
the operator algebra version).

\begin{theorem}[\cite{Step70, Step73, WW88,W08}]\label{thm:tri_edges}
Let $\gA$ be an operator ideal, $\cP$ a generating topology and $\cM$ a generating bornology on
Banach spaces.
We have
\begin{enumerate}
    \item $\sT(\gA) = \{\sT(\gA)(X) : X \in \sC\}$ is a generating topology.
    \item $\sB(\gA) = \{\sB(\gA)(Y) : Y \in \sC\}$ is a generating bornology.
    \item $\sO(\cP) = \{\sO(\cP)(X,Y) : X,Y \in \sC\}$ is an operator ideal.
    \item $\sO(\cM) = \{\sO(\cM)(X,Y) : X,Y \in \sC\}$ is an operator ideal.
    \item $\cP^\circ = \{\cP^\circ(Y) : Y \in \sC\}$ is a generating bornology.
    \item $\cM^\circ = \{\cM^\circ(Y) : Y \in \sC\}$ is a generating topology.
    \item $\sO(\sT(\gA)) = \gA^\inj$.
    \item $\sO(\sB(\gA)) = \gA^\sur$.
    \item $\sT(\sO(\cP))  = \cP$.
    \item $\sB(\sO(\cM))  = \cM$.
\end{enumerate}
\end{theorem}

For generating topologies $\cP$ and $\cP_1$, and generating
bornologies $\cM$ and $\cM_1$ on Banach spaces,  we can also associate to them
the operator ideals with components
\begin{align*}
\sO(\cP/\cP_1)(X,Y) &= \gL(X_\cP,Y_{\cP_1}),\\
\sO(\cM/\cM_1)(X,Y) &=L^\times(X^{\cM_1},Y^{\cM})\cap\gL(X,Y)\\
\sO(\cP/\cM)(X,Y) &= \gL(X_\cP,Y^\cM).
\end{align*}
These give rise to
\begin{align*}
\sO(\cP/\cP_1) &= \sO(\cP_1)^{-1}\circ\sO(\cP)\quad \text{\cite{Step80}},\\
\sO(\cM/\cM_1) &= \sO(\cM)\circ\sO(\cM_1)^{-1} \quad \text{\cite{Step83}},\\
\sO(\cP/\cM) &= \sO(\cM)\circ\sO(\cP) = \sO(\cM)\circ\sO(\cP)\quad \text{\cite{WW88}}.
\end{align*}
Here, the product $\gA\circ\gB$ consists of operators of the form $ST$ with
$S\in \gA$ and $T\in \gB$, while the quotient $\gA\circ\gB^{-1}$ consists
of those $S$ such that $ST\in \gA$ whenever $T\in \gB$.
 Readers are referred to Pietsch's classic \cite{Pie80}
 for information regarding quotients and products of operator ideals.

\section{Generating sequences}

The following is based on a version in \cite{Step80}.

\begin{definition}
A family $\Phi(E)$ of bounded sequences in a Banach space $E$ is
    called a \emph{vector sequential structure} if
    \begin{itemize}
        \item $\{ax_n + by_n\}\in \Phi(E)$ whenever $\{x_n\},\{y_n\}\in \Phi(E)$ and $a,b$ are scalars;
        \item every subsequence of a sequence in $\Phi(E)$ belongs to $\Phi(E)$.
    \end{itemize}
\end{definition}

Clearly, the biggest vector sequential structure on a Banach space $E$ is the family
$\Phi_b(E)$ of all bounded sequences in $E$.  On the other hand, we let
$\Phi_{\text{fin}}(E)$ to be the vector sequential structure on $E$ of all
bounded sequences $\{x_n\}$ with finite dimensional ranges, i.e., there is a finite
dimensional subspace of $F$ containing all
$x_n$'s.

Given two Banach spaces $E, F$ with vector sequential structures $\Phi, \Psi$, we let
$\gL(E^\Phi, F^\Psi)$ be the vector space of all continuous linear operators sending
a bounded sequence in $\Phi$ to a bounded sequence in $F$ with a subsequence in $\Psi$.

\begin{definition}
A family
$\Phi := \{\Phi(F) : F \text{ is a Banach space}\}$
of (bounded) vector sequential structures on each Banach space is called a \emph{generating sequential
structure} if
    \begin{description}
        \item[GS$_1$] $\Phi_{\text{fin}}(F) \subseteq \Phi(F) \subseteq \Phi_b(F)$ for every Banach space $F$.
        \item[GS$_2$] $\gL(E,F)\subseteq \gL(E^\Phi, F^\Phi)$ for all Banach spaces $E,F$.
    \end{description}
\end{definition}

\begin{definition}
Let $\gA$ be an operator ideal on Banach spaces.
Call a bounded sequence $\{y_n\}$ in a Banach space $F$ an \emph{$\gA$-sequence}
if there is a bounded sequence $\{x_n\}$ in a Banach space $E$ and a continuous
linear operator $T$ in $\gA(E,F)$ such that $Tx_n=y_n$, $n=1,2,\ldots$.
Denote by $\sS(\gA)(F)$ the family of all $\gA$-sequences in $F$.
\end{definition}

\begin{lemma}\label{lem:A-seq}
Let $\gA$ be an operator ideal on Banach spaces.
Then
$$
\sS(\gA) := \{\sS(\gA)(F): F \text{ is a Banach space}\}
$$
is a generating
sequential structure.
\end{lemma}
\begin{proof}
It is easy to see that all $\sS(\gA)(F)$ is a bounded vector
sequential structure.  Since the ideal $\gF$ of continuous linear operators  of finite rank
is the smallest operator ideal, while the ideal $\gL$ of all continuous
linear operators is the biggest, we see that (GS$_1$) is satisfied.
On the other hand, (GS$_2$) follows from (OI$_2$) directly.
\end{proof}

We now work on the converse of Lemma \ref{lem:A-seq}.
Let $\Phi$ be a generating sequential structure.
Denote by $\sM^\Phi_\text{base}(E)$ the family of all bounded sets $M$ in a
Banach space $E$ such that every sequence $\{x_n\}$ in $M$ has a subsequence
$\{x_{n_k}\}$ in $\Phi(E)$.

\begin{lemma}\label{lem:seq2born}
The family $\sM^\Phi(E)$ of all (bounded) sets with circle hull in
$\sM^\Phi_\text{base}(E)$ in a Banach space $E$ forms a generating bornology $\sM^\Phi$.
\end{lemma}
\begin{proof}
Let $M, N\in \sM^\Phi(E)$, let $\lambda$ be a scalar, and let $W\subseteq M$.  It is
easy to see that the circle hull $\text{ch}\, M$, $\lambda M$ and $W$ are all
in $\sM^\Phi(E)$.  For the sum $M+N$, we first notice that both
$\text{ch}\, M$ and $\text{ch}\, N$ belong to $\sM^\Phi_\text{base}(E)$.  Now
$$
M+N \subseteq \text{ch}\, M + \text{ch}\,  N \in \sM^\Phi_{\text{base}}(E)
$$
and
$$
\text{ch}\, (M+N) \subseteq \text{ch}\, M + \text{ch}\, N
$$
give $M+N\in \sM^\Phi(E)$.  All these together say that
$\sM^\Phi(E)$ is a vector bornology on the Banach space $E$.

The condition (GB$_1$) follows directly from (GS$_1$).  For (GB$_2$),
let $T\in \gL(E,F)$ and $M\in \sM^\Phi(E)$.
Then for every sequence $\{Tx_n\}$ in $TM$, we
will have a subsequence of $\{x_n\}$ in $\Phi(E)$.
By (GS$_2$), we have a further subsequence $\{Tx_{n_k}\}$ belonging
to $\sM^\Phi(F)$.
\end{proof}

\begin{definition}
A generating sequential structure $\Phi$ is said to be \emph{normal}
if for every $\{x_n\}$ in $\Phi(E)$ and every bounded scalar sequence $(\lambda_n)$ in $\ell^\infty$
of norm not greater than $1$, we have $\{\lambda_n x_n\}\in \Phi(E)$.
\end{definition}

It is clear that $\sS(\gA)$ is normal for any operator ideal $\gA$.

\begin{lemma}\label{lem:base}
For every normal generating sequential structure $\Phi$, we have
$$
\sM^\Phi_{\rm{base}} = \sM^\Phi.
$$
\end{lemma}
\begin{proof}
It suffices to show that every $M$ in $\sM^\Phi_{\text{base}}(E)$
has its circle hull $\text{ch}\, M$ in $\sM^\Phi_{\text{base}}(E)$ for
each Banach space $E$, i.e., $\sM^\Phi_{\text{base}}$ is itself circled.
To this end, let $\{x_n\}$ be from $\text{ch}\, M$.  Then there is
a $(\lambda_n)$ from $\ell^\infty$ of norm not greater than $1$ such that
$$
x_n = \lambda_n y_n, \quad n=1,2,\ldots
$$
for some $\{y_n\}$ in $M$.  Now, there is a subsequence $\{y_{n_k}\}$ in $\Phi(E)$
ensuring that $\{x_{n_k}\} = \{\lambda_{n_k} y_{n_k}\} \in \Phi(E)$ for
$\Phi$ being normal.
\end{proof}

\begin{definition}
Let $\Phi$ be a generating sequential structure on Banach spaces.
Let $E,F$ be Banach spaces.
Denote by $\sO(\Phi)(E,F)$ the family of all continuous linear operators $T$
in $\gL(E,F)$ sending every bounded sequence $\{x_n\}$ in $E$ to a bounded
sequence in  $F$ with a subsequence $\{Tx_n\}$ in $\Phi(F)$.  In other words,
$\sO(\Phi)(E,F) = \gL(E^{\Phi_b}, F^\Phi)$.
\end{definition}

\begin{theorem}\label{thm:seq2OI}
Let $\Phi$ be a generating sequential structure on Banach spaces.
Then $\sO(\Phi)$ is a surjective operator ideal.  More precisely,
$$
\sO(\Phi)= \sO(\sM^\Phi).
$$
\end{theorem}
\begin{proof}
The ideal property of $\sO(\Phi)$ is trivial.  For the surjectivity of
$\sO(\Phi)$, we let $S\in \sO(\Phi)(E_0,F)$ and $T\in\gL(E,F)$ such that
$TU_E\subseteq SU_{E_0}$.  Here, $U_E, U_{E_0}$ are the closed unit balls of
$E, E_0$, respectively.  We need to show that $T\in \sO(\Phi)(E,F)$.  In fact,
for any bounded sequence $\{x_n\}$ in $E$, we have
$$
Tx_n = Sy_n, \quad n=1,2,\ldots,
$$
with some bounded sequence $\{y_n\}$ in $F$.  Now, the fact $\{Sy_n\}$ has
a subsequence $\{Sy_{n_k}\}\in \Phi(F)$ implies that $T\in \sO(\Phi)(E,F)$, as
asserted.

Finally, we check the representation.  The inclusion $\sO(\Phi)\subseteq \sO(\sM^\Phi)$ is
an immediate consequence of the definition of $\sM^\Phi$.  Conversely, let $T\in \sO(\sM^\Phi)(E,F)$,
i.e., $TU_E\in \sM^\Phi(F)$.  Let $\{x_n\}$ be a bounded sequence in $E$.
We can assume $\|x_n\|\leq 1$ for all $n$.  Then $\{Tx_n\}\subseteq TU_E$ and hence
possesses a subsequence $\{Tx_{n_k}\}\in \Phi(F)$.  In other words, $T\in \sO(\Phi)(E,F)$.
\end{proof}

\begin{corollary}
If a generating sequential structure $\Phi$ is normal then
$$
\sS(\sO(\Phi))=\Phi.
$$
Conversely, if $\gA$ is an operator ideal on Banach spaces then
$$
\sO(\sS(\gA)) = \gA^\sur.
$$
\end{corollary}

\begin{theorem}\label{thm:quotient}
Let $\Phi,\Psi$ be two generating sequential structures.
Let
$$
\sO(\Phi/\Psi)(E,F) := \gL(E^\Psi, F^\Phi),
$$
i.e., $T\in \sO(\Phi/\Psi)(E,F)$
if and only if the continuous linear operator $T$ sends each bounded sequence
$\{x_n\}$ in $\Psi(E)$ to a sequence in $F$ with a subsequence $\{Tx_{n_k}\}$ in
$\Phi(F)$.  Then $\sO(\Phi/\Psi)(E,F)$ is an operator ideal on Banach spaces.  Moreover,
if $\Psi$ is normal, then
$$
\sO(\Phi/\Psi)= \sO(\sM^\Phi/\sM^\Psi)= \sO(\sM^\Phi)\circ\sO(\sM^\Psi)^{-1}.
$$
\end{theorem}
\begin{proof}
Conditions (OI$_1$) and (OI$_2$) follow from Conditions (GS$_1$) and (GS$_2$), respectively.

Assume $\Psi$ is normal.  Let $T\in \sO(\Phi/\Psi)(E,F)$ and $M\in \sM^\Psi(E)$ be circled.
Suppose, on contrary, $TM\notin \sM^\Phi(F)$.  By definition, there is a sequence $\{Tx_n\}$
in $TM$ having no subsequence $\{Tx_{n_k}\}$ in $\Phi(F)$.  But since $\{x_n\}\subseteq M$,
it has a subsequence $\{x_{n_k}\}$ in $\Psi(E)$.  And thus $\{Tx_{n_k}\}$ has a further subsequence
in $\Phi(F)$, a contradiction.

Conversely, let $T\in\sO(\sM^\Phi/\sM^\Psi)(E,F)$.  It is easy to see that
the range of any sequence $\{x_n\}$ in $\Psi(E)$ is a bounded set in $\sM^\psi(E)$, by
noting that $\Psi$ is normal.  The range of the sequence $\{Tx_n\}$ is bounded in $\sM^\Phi(F)$.
Therefore, there is a subsequence $\{Tx_{n_k}\}$ of $\{Tx_n\}$ belongs to $\Phi(F)$.
This ensures $T\in \sO(\Phi/\Psi)(E,F)$.
\end{proof}

\begin{definition}\label{defn:coordinated}
Given two generating sequential structures $\Phi$ and $\Psi$.  We say
that $\Psi(E)$ is \emph{coordinated to} $\Phi(E)$ if the following conditions holds.
Suppose $\{x_n\}\in \Psi(E)$.  The bounded sequence $\{x_n\}\notin \Phi(E)$ if and only if
$\{x_n\}$ has a subsequence $\{x_{n_k}\}$ which has no further subsequence belonging to $\Phi(E)$.
We write, in this case, $\Phi(E) \ll \Psi(E)$.  If $\Phi(E) \ll \Psi(E)$ for all Banach spaces
$E$ then we write $\Phi \ll \Psi$.
\end{definition}

\begin{theorem}\label{thm:coordinated-quotient}
Assume $\Phi,\Psi$ be two generating sequential structures such that $\Phi(F)\ll\Psi(F)$ for
some Banach space $F$.  Then $\sO(\Phi/\Psi)(E,F)$ consists of exactly those continuous
linear operators sending bounded sequences in $\Psi(E)$ to bounded sequences in $\Phi(F)$ for
any Banach space $E$.
\end{theorem}
\begin{proof}
Assume $T\in \sO(\Phi/\Psi)(E,F)$ and $\{x_n\}\in \Psi(E)$.  Now, suppose
$\{Tx_n\}\notin \Phi(F)$.  Since $\{Tx_n\}\in \Psi(F)$ by (GS$_2$), we
can verify the condition in Definition \ref{defn:coordinated}.
Therefore, we might have a subsequence $\{Tx_{n_k}\}$ having no further subsequence in
$\Phi(F)$.  However, by the facts that $\{x_{n_k}\}\in \Psi(E)$ and $T\in \sO(\Phi/\Psi)(E,F)$,
we arrive at the asserted contradiction.
\end{proof}

\begin{corollary}
Let $\Phi,\Psi$ be two generating sequential structures such that $\Phi \ll\Psi $ and
$\Psi$ is normal.  Then, the operator ideal $\sO(\sM^\Phi, \sM^\Psi)$ consists
of exactly those operators sending sequences in $\Psi$ to   sequences in $\Phi$.
\end{corollary}

\section{Examples}

We begin with some useful generating bornologies.

\begin{examples}
A set $W$ in a Banach space $F$ is called
\begin{enumerate}
    \item \emph{nuclear} (see, e.g., \cite{Step80}) if
    $$
    W \subseteq \{ y\in F: y = \sum_{n=1}^\infty \lambda_n y_n, |\lambda_n| \leq 1\}
    $$
    for an absolutely summable sequence $\{y_n\}$ in $F$;

    \item \emph{unconditionally summable} (see, e.g., \cite{Step80}) if
    $$
    W \subseteq \{ y\in F: y = \sum_{n=1}^\infty \lambda_n y_n \text{ in norm}, |\lambda_n| \leq 1\}
    $$
    for an unconditionally  summable series $\sum_{n=1}^\infty y_n$ in $F$;

    \item \emph{weakly unconditionally summable} (see, e.g., \cite{Step80}) if
    $$
    W \subseteq \{ y\in F: y = \sum_{n=1}^\infty \lambda_n y_n \text{ weakly}, |\lambda_n| \leq 1\}
    $$
    for a weakly unconditionally  summable series $\sum_{n=1}^\infty y_n$ in $F$;

    \item \emph{limited} (\cite{BourgainDiestel84}) if
    $$
    \lim_{n\to\infty} \sup_{a\in W} |y_n'(a)| = 0
    $$
    for any $\sigma(F',F)$-null sequence $\{y_n'\}$ in $F'$.
\end{enumerate}
The above defines generating bornologies $\sM_\nu$, $\sM_{uc}$, $\sM_{wuc}$, and $\sM_{\text{lim}}$, respectively.
\end{examples}

Another way to obtain generating bornologies is through  generating
sequential structures.

\begin{examples}
A bounded sequence $\{x_n\}$ in a Banach space $F$ is called
    \begin{enumerate}
        \item \emph{$\delta_0$-fundamental} (\cite{Step80}) if
        $$
        t_n = x_{n+1} - x_n, \quad n=1,2,\ldots,
        $$
        forms a weakly unconditionally summable series;

        \item \emph{limited} (\cite{BourgainDiestel84}) if
        $$
        \langle x_n, x_n'\rangle \to 0 \text{ as } n\to\infty
        $$
        for any $\sigma(F', F)$-null sequence $\{x_n'\}$ in $F'$;

        \item \emph{Banach-Saks} if
        $$
        \lim_{n\to\infty} \frac{x_1 + x_2 + \cdots + x_n}{n} = x_0
        $$
        for
        some $x_0$ in $F$ in norm.
    \end{enumerate}
\end{examples}

\begin{examples}
The following are all generating sequential structures.
\begin{enumerate}
    \item $\Phi_{\delta_0}$ of all $\delta_0$-fundamental sequences.
    \item $\Phi_{\text{lim}}$ of all limited sequences.
    \item $\Phi_{\text{BS}}$ of all Banach-Saks sequences.
    \item $\Phi_c$ of all convergent sequences.
    \item $\Phi_{wc}$ of all weakly convergence sequences.
    \item $\Phi_{wCa}$ of all weakly Cauchy sequences.
\end{enumerate}
\end{examples}

\begin{examples}[\cite{Step80}]
The following generating bornologies can be induced by the corresponding generating
sequential structures as in Lemma \ref{lem:seq2born}.
\begin{enumerate}
    \item The $\delta_0$-compact bornology $\sM_{\delta_0}$ is defined by $\Phi_{\delta_0}$.
    \item The compact bornology $\sM_c$ is defined  by $\Phi_c$.
    \item The weakly compact bornology $\sM_{wc}$ is defined  by $\Phi_{wc}$.
    \item The Rosenthal compact bornology $\sM_{wCa}$ is defined  by $\Phi_{wCa}$.
    \item The Banach-Saks bornology $\sM_{\text{BS}}$  is defined by $\Phi_{\text{BS}}$.
\end{enumerate}
\end{examples}

\begin{proposition}
The generating bornology induced by $\Phi_{\rm{lim}}$ is $\sM_{\rm{lim}}$.
\end{proposition}
\begin{proof}
First, we notice that $\Phi_{\text{lim}}$ is normal.  Hence,
$\sM_{\text{base}}^{\Phi_{\text{lim}}} = \sM^{\Phi_{\text{lim}}}$ by
Lemma \ref{lem:base}.  Let $W\in \sM_{\text{lim}}(F)$.  Then,
by definition,
$$
\lim_{n\to\infty} \sup_{a\in W} |\langle a,y_n'\rangle| = 0
$$
for any $\sigma(F',F)$-null sequence $\{y_n'\}$ in $F'$ and any
sequence $\{y_n\}$ in $W$.  Hence, $W\in \sM^{\Phi_{\text{lim}}}(F)$.

Conversely, if $W$ is a bounded set in $F$ such that every sequence
$\{y_n\}$ in $W$ has a limited subsequence $\{y_{n_k}\}$, we need
to check that $W\in \sM_{\text{lim}}$.  Assume, on contrary, that there
were some $\sigma(F',F)$-null sequence $\{y_n'\}$ in $F'$ and some sequence $\{a_n\}$ in $W$
such that
$$
\lim_{n\to\infty} |\langle a_{n_k}, y_{n_k}'\rangle| = 0.
$$
This is a contradiction and thus $W$ is a member of $\sM_{\text{lim}}(F)$.
\end{proof}

\begin{proposition}
The generating bornologies $\sM_c$, $\sM_{wc}$, $\sM_{wCa}$, $\sM_\nu$,
$\sM_{uc}$, $\sM_{wuc}$ and $\sM_{\rm{lim}}$ are all $\sigma$-disked.
\end{proposition}
\begin{proof}
For the case of $\sM_{\text{lim}}$, see \cite{BourgainDiestel84}.
For all others, see \cite{Step80}.
\end{proof}

We are now ready to give a number of examples.

\begin{examples}
Using Definition \ref{defn:vertices} and Theorem \ref{thm:seq2OI}, we can
obtain the following operator ideals on Banach spaces through the
associated generating bornologies.
\begin{enumerate}
    \item The ideal \emph{$\gN^\sur$ of co-nuclear operators} for the nuclear
    bornology $\sM_\nu$.
    \item The ideal \emph{$\gL_{\text{lim}}$ of limited operators} for the limited bornology
    $\sM_{\rm{lim}}$.
    \item The ideal \emph{$\mathfrak{BS}$ of Banach-Saks operators} for the Banach-Saks bornology
    $\sM_{\text{BS}}$.
    \item The ideal \emph{$\gK$ of compact operators} for the compact bornology $\sM_c$.
    \item The ideal \emph{$\gW$ of weakly compact operators} for the weakly compact bornology $\sM_{wc}$.
    \item The ideal \emph{$\mathfrak R$ of Rosenthal compact operators} for the Rosenthal compact bornology $\sM_{wCa}$.
    \item The ideal \emph{$\mathfrak{D}_0$ of $\delta_0$-compact operators} for the $\delta_0$-compact
    bornology $\sM_{\delta_0}$.
\end{enumerate}
For a proof of (a), see \cite[p.~112]{Pie80}.  For (b), see \cite{Drewnowski86}.
For (c), see Theorem \ref{thm:seq2OI}.  For others, see \cite{Step80}.
\end{examples}

Concerning Theorem \ref{thm:coordinated-quotient}, we have

\begin{examples}
\begin{enumerate}
    \item $\Phi_c \ll \Phi_{wc}$.
    \item $\Phi_c \ll \Phi_{wCa}$.
    \item $\Phi_{wc} \ll \Phi_{wCa}$.
    \item $\Phi_c \ll \Phi_{\delta_0}$.
    \item $\Phi_c \ll \Phi_{\text{lim}}$.
\end{enumerate}
The first four can be found in \cite{Step80}.
For the last one, we observe the fact that each limited sequence is weakly Cauchy (see \cite{BourgainDiestel84})
and (c).
\end{examples}

\begin{examples}
The following operator ideals have  desirable representations.
\begin{enumerate}
    \item The ideal $\mathfrak U$ of unconditionally summing operators
    \begin{align*}
    {\mathfrak U} &= \sO(\sM_{uc}/\sM_{wuc}) = \sO(\sM_{uc})\circ\sO(\sM_{wuc})^{-1}\\
        &= \sO(\sM_c/\sM_{\delta_0}) = \sO(\sM_c)\circ\sO(\sM_{\delta_0})^{-1} = \gK\circ{\mathfrak D}_0^{-1}.
    \end{align*}

    \item The ideal $\mathfrak V$ of completely continuous operators
    \begin{align*}
    {\mathfrak V} &= \sO(\Phi_c/\Phi_{wc}) = \sO(\sM_c/\sM_{wc}) = \sO(\sM_c)\circ\sO(\sM_{wc})^{-1} = \gK\circ\gW^{-1}\\
        &= \sO(\Phi_c/\Phi_{wCa}) = \sO(\sM_c/\sM_{wCa}) = \sO(\sM_c)\circ\sO(\sM_{wCa})^{-1}=\gK\circ
        {\mathfrak R}^{-1}.
    \end{align*}

    \item The ideal $w{\mathfrak{SC}}$ of weakly sequentially complete operators
    $$
    s{\mathfrak{SC}} = \sO(\Phi_{wc}/\Phi_{wCa}) = \sO(\sM_{wc}/\sM_{wCa}) =
        \sO(\sM_{wc})\circ\sO(\sM_{wCa})^{-1} = \gW \circ {\mathfrak R}^{-1}.
    $$

    \item The ideal $\mathfrak{GP}$ of Gelfand-Phillips operators
    $$
    {\mathfrak{GP}} = \sO(\Phi_c/\Phi_{\text{lim}})
                 = \sO(\sM_c/\sM_{\text{lim}}) = \sO(\sM_c)\circ\sO(\sM_{\text{lim}})^{-1}
                 = \gK\circ {\mathfrak L}_{\text{lim}}^{-1}.
    $$

    \item The ideal $w{\mathfrak{BS}}$ of weakly Banach-Saks operators
    $$
    w{\mathfrak{BS}} = \sO(\Phi_{\text{BS}}/\Phi_{wc})
                    = \sO(\sM_{\text{BS}}/\sM_{wc}) =  \sO(\sM_{\text{BS}})\circ\sO(\sM_{wc})^{-1}
                     = {\mathfrak{BS}}\circ\gW^{-1}.
    $$

\end{enumerate}
See \cite{Step80} for a proof of (a), (b) and (c).  See \cite{Drewnowski86} for (d).
For (e), we simply recall that a $T$ in $\gL(E,F)$ is called \emph{weakly Banach-Saks} if $T$
sends each weakly convergent sequence to a sequence possessing a Banach-Saks subsequence, by definition.
\end{examples}

\begin{proposition}[\cite{Step80}]
Let $\gA$ be a surjective operator ideal such that $\gA$ is idempotent, i.e., $\gA^2=\gA$.  Then
for any operator ideal $\gB\supseteq \gA$, there exists an operator
ideal $\gC$ such that
$$
\gA = \gC\circ \gB^\sur,
$$
and $\gC$ can be chosen to be $\gA\circ(\gB^\sur)^{-1}$.
\end{proposition}
\begin{proof}
We present a proof here for completeness.
Denote
$$
\sM_\gA := \sM(\gA), \quad \sM_\gB :=\sM(\gB)
$$
and set
$$
\gC := \sO(\sM_\gA/\sM_\gB).
$$
Now the facts $\gA\subseteq \gB$ and
$$
\gA\subseteq \gC =\sO(\sM_\gA)\circ\sO(\sM_\gB)^{-1} = \gA\circ(\gB^\sur)^{-1}
$$
implies that
$$
\gA = \gA^2 \subseteq \gC\circ\gB \subseteq \gC\circ\gB^\sur.
$$
On the other hand, if $T\in \gC\circ\gB^\sur(E,F)$ then $T=RS$ for
some $R\in \gC(G,F)= L^\times(G^{\sM_\gB},F^{\sM_\gA})\cap\gL(G,F)$ and $S\in\gB^\sur(E,G)=
L^\times(E,G^{\sM_\gB})\cap\gL(E,G)$ with some Banach space $G$.
Hence,
$$
T\in L^\times(E,F^{\sM_\gA})\cap\gL(E,F) = \gA(E,F).
$$
\end{proof}

\begin{examples}[\cite{Step80}]
\begin{enumerate}
    \item Since the surjective ideal $\gK$ of compact operators is idempotent and
    contained in the surjective ideal $\gW$, $\mathfrak R$, and ${\mathfrak D}_0$ of
    weakly compact operators, Rosenthal compact operators, and $\delta_0$-compact
    operators, respectively, we have
    \begin{enumerate}
        \item $\gK = {\mathfrak V}\circ \gW$ since ${\mathfrak V} = \sO(\sM_c/\sM_{wc})
        =\gK\circ\gW^{-1}$;
        \item $\gK = {\mathfrak V}\circ {\mathfrak R}$ since ${\mathfrak V} = \sO(\sM_c/\sM_{wcA})
        = \gK\circ{\mathfrak R}^{-1}$;
        \item $\gK = {\mathfrak U}\circ{\mathfrak D}_0$ since
        ${\mathfrak U} = \sO(\sM_c/\sM_{\delta_0})= \gK\circ{{\mathfrak D}_0}^{-1}$.
    \end{enumerate}
    \item Since the ideal $\gW$ of weakly compact operators is idempotent and contained in the surjective
    ideal $\mathfrak R$ of Rosenthal compact operators, we have
    $$
    \gW= w{\mathfrak{SC}}\circ {\mathfrak R},
    $$
    where $w{\mathfrak{SC}}=\sO(\sM_{wc}/\sM_{wCa}) = \gW\circ{\mathfrak R}^{-1}$ is the ideal
    of weakly sequentially complete operators.
\end{enumerate}
\end{examples}

\end{document}